\documentclass{amsart}
\usepackage[utf8]{inputenc}
\usepackage{algorithm}
\usepackage{algpseudocode}
\usepackage{float}
\usepackage{hyperref}

\newtheorem{theorem}{Theorem}
\newtheorem{corollary}{Corollary}

\newcommand{\lcm}{ {\rm lcm}}

\algnewcommand\algorithmicforeach{\textbf{for each}}
\algdef{S}[FOR]{ForEach}[1]{\algorithmicforeach\ #1\ \algorithmicdo}

\title{Advances in Tabulating Carmichael Numbers}
\author{Andrew Shallue}
\author{Jonathan Webster}
\address[1]{Illinois Wesleyan University}
\address[2]{Butler University}

\begin{document}

\begin{abstract}
We report that there are $49679870$ Carmichael numbers less than $10^{22}$ which is an order of magnitude improvement on Richard Pinch's prior work.  We find Carmichael numbers of the form $n = Pqr$ using an algorithm bifurcated by the size of $P$ with respect to the tabulation bound $B$.   For $P < 7 \cdot 10^7$, we found  $35985331$ Carmichael numbers and $1202914$ of them were less than $10^{22}$.  When $P > 7 \cdot 10^7$,  we found $48476956$ Carmichael numbers less than $10^{22}$.  We provide a comprehensive overview of both cases of the algorithm.  For the large case, we show and implement asymptotically faster ways to tabulate compared to the prior tabulation.  We also provide an asymptotic estimate of the cost of this algorithm.  It is interesting that Carmichael numbers are worst case inputs to this algorithm.  So, providing a more robust asymptotic analysis of the cost of the algorithm would likely require resolution of long-standing open questions regarding the asymptotic density of Carmichael numbers.   
\end{abstract}

\maketitle


\section{Introduction}

Fermat's Little Theorem states that when $p$ is prime that $a^{p} \equiv a \pmod{p}$ for any integer $a$.  The converse of this theorem gives a computationally efficient way to detect if an integer is composite.  Pick an integer $a < n$ and compute $a^{n} \pmod{n}$.  If the result is not $a$, conclude that $n$ is composite.  Unfortunately, there are composite integers for which this test will fail to detect as composite, e.g. $n = 341$ and $a=2$.  Even worse, there are composite numbers for which $a^n \equiv a \pmod{n}$ for any $a$, and the least such example is $561$.  It would be reasonable to call these numbers \textit{absolute Fermat pseudoprimes} or \textit{\u{S}imerka numbers} since V. \u{S}imerka found the first seven examples in 1885 \cite{first_source}.  However, we will keep with convention and call them \textit{Carmichael numbers}.  

For a background and survey on Carmichael numbers we refer to \cite{carmsurvey}.  While that survey is a bit dated, it shows the breadth of questions that have been asked.  Our concern is tabulating Carmichael numbers and this specific problem has a long pedigree \cite{swift, pom_self, jaeschke_carm, keller, guthmann}.  The tabulation bounds for those sources were (respectively): $10^9$, $ 25 \cdot 10^9$, $10^{12}$, $10^{13}$, and $10^{14}$.   Our key source is R.G.E. Pinch's \textit{The Carmichael numbers up to $10^{15}$} \cite{pinch15}.  He also published a series of reports in which he announced further tabulations all using the same algorithm \cite{pinch16, pinch17, pinch18, pinch21}. 
 The final tabulation found all Carmichael numbers less than $10^{21}$.    We address the problem of tabulation by investigating the difficulty of finding primes $q$ and  $r$ given $P$ so that $Pqr$ is a Carmichael number.

We analyze this problem from two different perspectives:  \textit{bounded} and \textit{unbounded}.  In the former case, we construct $q$ and $r$ so that $Pqr < B$.  In the latter case, we analyse the cost to find any such primes so that $Pqr$ is a Carmichael number.  The unbounded problem is interesting on its own and arises naturally in the context of a bounded computation.  If $P$ is small enough with respect to $B$, any primes $q$ and $r$ for which $Pqr$ is a Carmichael number will inherently satisfy $Pqr < B$.

Our contributions are as follows.  We give a comprehensive overview and an analysis of tabulation algorithms.   The algorithm is bifurcated by size.  For the small case, we review the improvements introduced in \cite{sw_ants2}.  We do this for the sake of completeness and because the algorithms played a significant role in this tabulation.  In \cite{sw_ants2}, it was reported that the algorithm was run on $P < 3\cdot 10^6$ and now this has been extended to $P < 7\cdot 10^7$.  In contrast, this paper shows a better way to handle the large case.  In particular, there are specific cases in the tabulation which would have required an exponential amount of work but we give multiple ways to reduce this to only a polynomial amount of work (see Section \ref{sec:P_large}).  We also give a heuristic estimate on the amount of time to tabulate up to $B$, namely
\[ B\exp\left(\left(-\frac{1}{\sqrt{2}} +o(1)\right)\sqrt{\log B \log\log B}\right), \]
by providing an estimate on the number of preproducts the large case has to consider.  See Theorem \ref{num_preprod} in Section \ref{sec:analysis} for a precise statement and argument.  

In \cite{two_contra}, it is recounted how Shanks and Erd\H{o}s disagreed regarding the asymptotic count of Carmichael numbers.  Shanks had plausible beliefs informed by available empirical data.  Erd\H{o}s was informed by a heuristic arising from a theoretical construction technique.  The conflict between empirical data and theory is reflected in this very paper.  A practical algorithm for tabulation should be guided by empirically observation.  When we make arguments based off of this data, we will speak of a \textit{guide} or \textit{guidance}.  However, we are also informed by theory and do not want to give the impression that the guidance is a theorem (e.g., Subsection \ref{subsec:benefits}).   At the same time, we are concerned with asymptotic analysis of algorithms and offer asymptotic improvements whose benefits might not be empirically noticed in our current computational domain.  

While we remain far from proving Erd\H{o}s' heuristic, it is generally accepted as a valid model to understand the asymptotic count of Carmichael numbers.  His view gives a heuristic asymptotic count of 
\[ B \exp \left( \frac{ -k \log B \log \log \log B}{\log \log B} \right) \]
for some positive constant $k$.   This implies that the count is greater than  $ B^{1 - \epsilon}$ for any fixed $\epsilon >0$.  When Alfords, Pomerance, and Granville proved the infinitude of Carmichael numbers, they also established a lower bound  $B^{2/7}$ for their count\cite{inf_carm}.  The current best lower bound is $B^{1/3}$\cite{better_inf},  which is more empirically in line with what current tabulations show.  We seem quite far from the point at which tabulations could give empirical evidence for Erd\H{o}s' heuristic.  This dichotomy was also explicitly pointed out by Swift in 1975 \cite{swift}.

 The Erd\H{o}s construction starts  with a carefully selected composite number $L$, and builds Carmichael numbers around $L$.  There exist algorithmic versions of this construction technique that produced Carmichael numbers with billions of prime factors (e.g., \cite{subset}).  A tabulation algorithm should be able to account for how such a number would be found, even if only in a hypothetical way.  We show the sorts of numbers that arise from this construction technique play a role as worst-case inputs to our tabulation algorithm (see Subsection \ref{bad_problems}).  Thus, we would likely need the asymptotic count of these numbers to prove a strong bound on the algorithmic cost and we find it especially pleasing to see this problem show up in an intrinsic way in the algorithm.  If we were strictly guided by empirical observation it would have been more difficult to see this connection.  
 
The remainder of the paper is organized as follows.  In Section \ref{sec:background}, we review some general number-theoretic results that are needed and the specific results on Carmichael numbers.  In Section \ref{sec:alg_over}, we provide a high-level overview of the algorithm.  In Section \ref{sec:P_small}, we review the small case and discuss its benefits.  In Section \ref{sec:P_large}, we explain the large case. 
 This includes Pinch's approach and improvements based on \textit{divisors in residue classes} \cite{div_in_class, div_in_class_con}.  In Section \ref{sec:bndadmis}, we discuss how construction of preproducts differs between the small and the large cases.  In Section \ref{sec:analysis}, we provide a theorem on the count of preproducts.  This serves as a guide for the total amount of work our tabulation algorithm does.  Finally, we conclude with a report on our implementation and the tabulation to $B = 10^{22}$.  This independently verifies Pinch's prior tabulations and exceeds his by an order of magnitude.

\section*{Acknowledgements}  We thank Richard Pinch, Carl Pomerance, and Jon Sorenson for helpful discussions that have made this paper better.  We thank Drew Sutherland for his generosity and  the Simons Collaboration in Arithmetic Geometry,
Number Theory, and Computation for computational support.  We also thank Sungjin Kim for his help during a problem session at the 2022 West Coast Number Theory Conference.  Thanks also to Frank Levinson for generous donations to Butler University that funded the BigDawg compute cluster.  The second author is also grateful to the Holcomb Awards Committee for the financial support of this project.

\section{Background} \label{sec:background}

\subsection{Number Theoretic functions}
Let $\phi(n)$ be Euler's totient function.  Let $\lambda(n)$ be the Carmichael function or the reduced totient function.
We have
\begin{align*}
\lambda\left(p^{\alpha}\right) &= \begin{cases} \phi\left(p^{\alpha}\right) &\text{ if } \alpha \leq 2 \text{ or } p \geq 3 \\ \frac{1}{2} \phi\left(p^{\alpha}\right) &\text{ if } \alpha \geq 3 \text{ and } p = 2 \end{cases},
\end{align*}
\[ \lambda(n) = \lambda\left( {p_1}^{\alpha_1} \cdots {p_k}^{\alpha_k} \right) = \lcm \, \big(\lambda\left({p_1}^{\alpha_1}\right), \ldots, \lambda\left({p_k}^{\alpha_k}\right)\big), \]
where the $p_i$ are distinct prime divisors of $n$.  

This function shows up frequently in the tabulation algorithm.  We are influenced by Theorem 2 and 3 of \cite{lambda_cite} and Theorem 5 of \cite{lambda_cite2}.  A guiding and colloquial summary of these results might be, ``$\lambda(n)$ is frequently close to $n$ in size, and $\lambda(n)$ may be very small infrequently."  We forgo precise statements because they are not used explicitly and we leave the reader to consult the sources.  Even though we assume, as guidance, that $\lambda(n)$ is frequently close to $n$ in size, we design an algorithm that allows very small values of $\lambda(n)$.  Indeed, we will see exactly how and where the small values play a significant role in the algorithm.   

Let $\tau (n)$ be the function that counts the number of divisors of $n$.  Dirichlet established
\[ \sum_{ n \leq x } \tau(n) = x\log x + (2\gamma -1)x + O(\sqrt{x}) .\]

\subsection{Theorems for Carmichael numbers}

\begin{theorem}[Korselt's Criterion]
A composite number $n$ is a Carmichael number if and only if $n$ is squarefree and $(p-1) \mid (n-1)$ for all prime divisors $p$ of $n$.  
\end{theorem}

The tabulation methods construct $n$ in factored form.   Our goal will be to construct square-free odd numbers $n$  which will then be tested with Korselt's criterion.  Letting $d$ be the number of prime factors dividing $n$, then $d > 2$.  Let $P_k = \prod_{i = 1}^{k} p_i$ where $p_i < p_j$ iff $i < j$.  The primary tabulation methods concern $P_{d-2}$ and constructing or searching for $p_{d-1}$ and $p_{d}$.  Since these quantities are used so often, we write these quantities as $P$, $q$, and $r$.  For reference, this means that $P_{d-1} = Pq$ and $P_d = Pqr$. 

\begin{theorem}[Proposition 1 of \cite{pinch15}]
Let $n$ be a Carmichael number less than $B$.
\begin{enumerate}
\item[(2.1)] Let $k < d$.  Then $p_{k+1} < (B/P_{k})^{1/(d-k)}$ and $p_{k+1} - 1$ is relatively prime to $p_i$ for all $ i \leq k$.  
\item[(2.2)] We have $P_{d-1}r \equiv 1 \pmod{\lambda(P_{d-1})}$ and $r -1$ divides $P_{d-1} - 1$.  
\item[(2.3)] Each $p_i$ satisfies $p_i < \sqrt{n} < \sqrt{B}$.  
\end{enumerate}
\end{theorem}

The condition that $p_{k+1} - 1$ is relatively prime to $p_i$ for all $ i \leq k$ means that every $P$ is a cyclic number (in the group theory sense).  Letting $c(B)$ count the number of cyclic numbers less than $B$,  Erd\H{o}s proved \cite{erdos_cyclic} 
\[  c(B) \sim \frac{e^{-\gamma}B}{\log \log \log B}. \]
 Paul Pollack recently improved this result, proving that $c(B)$ admits an asymptotic series expansion in descending powers of $\log \log \log B$ terms with explicitly computable constants \cite{pollack_cyclic}.  Since we also require $Pp_{d-1}^2 < B$, we can consider a significantly smaller set of cyclic numbers.  We will address the impact of this condition in Sections \ref{sec:bndadmis} and \ref{sec:analysis}.

\begin{theorem}[Proposition 2 of \cite{pinch15}]   
\label{CDtheorem}
There are integers $2 \leq D < P < C$ such that, putting $\Delta = CD - P^2$, we have
\begin{eqnarray}
q = \frac{(P-1)(P+D)}{\Delta} + 1,\\
 r = \frac{(P-1)(P+C)}{\Delta} + 1,\\
 P^2 < CD < P^2\left( \frac{p_{d-2} + 3}{p_{d-2} + 1} \right).
 \end{eqnarray}
\end{theorem}

\begin{corollary}
There are only finitely many Carmichael numbers of the form $Pqr$ for a given $P$.  
\end{corollary}

The immediately above theorem and corollary were the primary motivation of \cite{sw_ants2}.  Given the finiteness result, it makes sense to tabulate with respect to $P$.  Once a given $P$ is accounted for, it never needs to be revisited.  The permanence of the tabulation may help reduce the work of future tabulations.  As it stands, the below sections show that the above theorem is used in an algorithmic way when $P$ is accounted as small with respect to the total tabulation bound.  

\begin{theorem}[Proposition 3 of \cite{pinch15}]
We may bound the size of the last two primes $q$ and $r$.  
\begin{enumerate}
    \item $q < 2P^2$
    \item $r < P^3$
\end{enumerate}

\end{theorem}

The computational cost of an unbounded direct search is $\Omega(P^2)$ due to the cost of finding $q$.   It is difficult to pin down the cost of finding $r$ given $P$ and $q$ due to the role $\lambda(P_{d-1})$ plays.  We say more in Section \ref{sec:P_large}.  An implementation will always have access to $\lambda(P_{d-1})$ because $P_{d-1}$ is always constructed in factored form.

\subsection{Computational Model}

We are concerned with the counts of arithmetic operations used to generate triples $(P, q, r)$ and do not concern ourselves with the asymptotic cost of verifying if $Pqr$ is a Carmichael number.  The primary reason we omit the cost of verification is that the cost is not uniform across the computation.  We will treat the cost of verification as being done in constant time as a guide.  

The primality test is the most expensive part of the Korselt check and we deferred it until other checks have passed.  This includes integrality checks for $q$ and $r$ and the $d < 14$ divisibility statements in the Korselt criterion. We may also discard candidates beyond the tabulation bound.  Even when all of these pass, and we need to verify the primality by use of a more expensive test, in some cases we know the complete (or partial) factorizations of $q-1$ or $r-1$ and have faster primality tests available.

We will account for storage costs using the bit model.  And, as is typical in tabulation algorithms, we do not count the cost of storing the output.

\section{Algorithmic Overview}\label{sec:alg_over}

For reference, we provide a high-level overview of the algorithm that is described across several sections below.  We want to analyze each piece, justify our choices, and compare to the previous algorithms.   We conclude with an overall asymptotic analysis in Section \ref{sec:analysis}.  

\begin{algorithm}[H]  
\renewcommand{\thealgorithm}{}
\caption{Tabulating Carmichael numbers less than $B$}\label{alg:cap}
\begin{algorithmic}
\Require $0 < X < B $, $X$ a cross-over point  \Comment{See Section \ref{subsec:benefits}.}
\ForEach{$P < X$} \Comment{See Section \ref{subsec:P}.}
    \If{ $P_{d-3}$ is small}
        \State Use Hybrid method to find $q$ and $r$.  \Comment{See Sections \ref{subsec:cd}-\ref{subsec:hybrid}.}
    \Else
        \State Use $D\Delta$ method to find $q$ and $r$.  \Comment{See Section \ref{subsec:ddelta}.}
    \EndIf
\EndFor
\ForEach{$P$ satisfying $P \geq X $ and $Pp_{d-2}^2 < B$ } \Comment{See Section \ref{subsec:Pdm1}.}
    \ForEach {$q$ an admissible prime in $(p_{d-2}, \sqrt{B/P})$}
    \If{ $B/(Pq\lambda(Pq))$ or $Pq/\lambda(Pq)$ is small }
       \State Compute $r$ by sieving.   \Comment{See Section \ref{subsec:bnd_lrg}.}
     \Else
       \State Factor $Pq-1$ using divisors in residue classes.  \Comment{See Section \ref{subsec:faster}} 
    \EndIf
    \EndFor
\EndFor
\end{algorithmic}
\end{algorithm}

\section{When \texorpdfstring{$P$}{} is small} \label{sec:P_small}

We present a concise and clear exposition of the key ideas found in \cite{sw_ants2}.  In particular, the proofs of the run-times make it a bit more clear what the implied constants are for the computation associated for a given $P$.  

\subsection{ The \texorpdfstring{$CD$}{}  method } \label{subsec:cd}
The starting point is the third paragraph of Section 2 of \cite{pinch15} where the small case is described.   
Theorem \ref{CDtheorem} is used to create the values $C$ and $D$ with a double nested loop.  
The outer loop is on $D$ and the inner loop is on $C$ permitted by inequality $(3)$ of Theorem \ref{CDtheorem}.  
With $C$, $D$, and $\Delta$, the inner loop does divisibility checks on $q$ and $r$, an optional bound check on $Pqr$,  Korselt divisibility checks, and two primality checks.

In \cite{sw_ants2}, it was shown that the number of $CD$ pairs created is akin to running a sieve of Eratosthenes on the interval found in the inequality $(3)$ of Theorem \ref{CDtheorem}.   Letting the length of that interval be $L_P$, then its value is about $2PP_{d-3}$.  Considering the $d=3$ case, we have $P_{d-3} = 1$ and so, $L_p$ is linear in $P$.  As $d$ grows larger, $L_P$ can approach quadratic in $P$ because $P_{d-3}$ can approach $P^{1 - \frac{1}{d-2}}$.

The number of $CD$ pairs created in this approach is 
\[O(L_P \log P) = O(P P_{d-3} \log P) = O(P^{2 - \frac{1}{d-2} }\log P).\]  The first equality is just rewriting $L_P$ in terms of $P$ and $P_{d-3}$, and would also be a $\Theta$ bound.  The second equality is a containment statement of big-$O$ costs, since it represents an upper bound on $P_{d-3}$.  This method performs poorly when $P_{d-3}$ is large with respect to $P$.    

\subsection{The \texorpdfstring{$D\Delta$}{} method} \label{subsec:ddelta}

The improvement found in \cite{sw_ants2}, is to replace the loop on $C$ with a loop on $\Delta$, constructed divisors of $(P-1)(P+D)$.  The divisors are obtained using a variant of the sieve of Eratosthenes.  The factorization of $(P-1)$ remains constant through all $D$ and the sieve variant factors all numbers in $[P+2, 2P-1]$,  which accounts for the factors of any $(P+D)$ term.  Even though the cost of obtaining the divisors exceeds the cost to increment $C$, the asymptotically fewer number of the divisors makes up for increased overhead used to enter the inner loop.  Like the inner loop of the $CD$ method, this approach does divisibility checks on $C = (P^2 + \Delta)/D$ and $r$, an optional bound check, Korselt divisibility checks, and two primality checks.  

\begin{theorem}
The number of $D \Delta$ pairs created is $O( \tau(P-1) P \log P)$.
\end{theorem}

\begin{proof}
For $P$ we run through $D$ in the interval $[2, P-1]$ and consider divisors of $(P-1)(P+D)$.  So we count the number of such divisors.  
\begin{eqnarray*}
\sum_{D=2}^{P-1} \tau \left( (P-1)(P+D ) \right) <& \displaystyle{  \tau(P-1)  \left(\sum_{ D = P}^{2P} \tau(D) \right) }\\
<& \displaystyle{  \tau(P-1)  \left(\sum_{ D < 2P}\tau(D)  - \sum_{ D < P}\tau(D)  \right) }\\
=& \tau(P-1)\left( P \log (P)+  O(P)\right) .
\end{eqnarray*}

\end{proof}

There are a few practical improvements to the above to further reduce the number of divisors constructed.  We may consider divisors of $(P-1)/2$ instead of $(P-1)$ because this forces $q$ to be an odd integer.  The set of divisors may be further reduced considering a small prime $p$ dividing $D$ and only generating $\Delta$ satisfying  $P^2 \equiv -\Delta \pmod{p}$ for small primes $p$ dividing $D$.  If $P \equiv 0,1 \pmod{p}$, we either always include or exclude $p$ from the odometer generating the divisors.  It is not too hard to see that both of these ideas only improve the implied constant on the big-$O$ cost.

\subsection{The Hybrid Method} \label{subsec:hybrid}

Since the outer loop on $D$ is the same in both methods, one may choose which inner loop is better to enter.  It is easy to compute the number of $C$ values or the number of divisors $\Delta$ that would have to be considered for a given $D$.  Roughly speaking, the choice will be between the smaller of $2PP_{d-3}/D$ or $\tau( (P-1)(P+D) )$.  While  $2PP_{d-3}/D$ decreases in a predictable way, the divisor count can behave more chaotically.   Generally, when $D$ is small the $D\Delta$ method will be chosen.  As $D$ gets larger, it may be that the $CD$ method becomes the faster way to complete the computation.  It was common for $P$ to be prime in our computational range; so, we found that implementing this hybrid method was helpful in reducing run-time.  In \cite{sw_ants2}, it is conjectured that the hybrid method considers $O(P \log \log P)$ $CD$ or $D\Delta$ pairs when $P$ is a prime.  If true, the hybrid method is an asymptotically faster algorithm for these cases.  

\subsection{Benefits and Crossover}\label{subsec:benefits}

The benefit of the $D \Delta$ method and the hybrid method is that it removes the dependence on $d$.  Another benefit that all three have is that they do not depend on the size of $\lambda(P)$.

It also makes deciding when to use these algorithms a bit easier.  There are two upper bounds on $q$ given $P$ and $B$.  On the one hand $q < 2P^2$ and on the other hand $Pq^2 < B$.  The cost of finding $q$ is roughly linear in the upper bound.  We can find $q$ (and $r$) by checking, on average,  $O( P (\log P)^2 )$ candidate pairs $D\Delta$.  Or we can find $q$ by exhaustive search on a precomputed lists of primes bounded by $\sqrt{B/P}$.  The number of elements in this list is bounded by 
\[ \pi(\sqrt{B/P}) \sim \frac{ \sqrt{B/P}}{ \log \sqrt{B/P}}. \] 
These two counts are equalized, roughly, at $P \sim B^{1/3}/(\log B)^2$.  In practice, we recommend a crossover of $X = B^{1/3}$.  Our own computation went a bit beyond this.  This was due, in part, because when we started the small case, we were unsure what the upper bound of the total tabulation was going to be.   As a guide, the cost of finding $r$ given $Pq$ is a polynomial time algorithm.  However, if the guidance is incorrect then choosing $X$ larger than $B^{1/3}$ would be justified.  

It is not clear to us if it is possible to provide an asymptotically faster algorithm for the small case by assuming a bounded computation.
 That is, as $P$ grows beyond $B^{1/6}$ it becomes increasingly likely that $Pqr > B$. 
 It would be good to have a modification that can incorporate this bound and results in an asymptotically faster algorithm.    

\section{When \texorpdfstring{$P$}{} is large}\label{sec:P_large} 

We show that the cost of completing individual preproducts $Pq$ may be exponential (see Theorem \ref{pinch_large}) in the worse case, but that many of these may be resolved in polynomial-time (see Subsections \ref{subsec:faster} and \ref{subsec:faster2} ).   Even the cases that may not be resolved in polynomial-time have asymptotically superior algorithms (see Theorem \ref{better_factoring}) compared to Pinch's approach.  The difficulty of finding $r$ given $P$ and $q$ is related to both $B$ and $\lambda(Pq)$.  In all of what follows, the goal is to use the two conditions on $r-1$: 

\begin{enumerate}
    \item $r-1$ divides $Pq-1$ and
    \item $r - 1 \equiv r^{\star} - 1 \pmod{ \lambda(Pq)}$ where $0 < r^{\star} < \lambda(Pq)$ is the modular inverse of $Pq \pmod{\lambda(Pq)}$.
\end{enumerate}

The first two subsections explain the simplest approaches to finding $r$;  both of which are described in \cite{pinch15}.   We prove that this can be an exponential time algorithm in $P$ (see Theorem \ref{pinch_large}).  The next two subsections establish several improvements each of which would change many of these individual exponential time computations to polynomial time ones.  Finally, we conclude with some comments on how Carmichael numbers that result from the Erd\H{o}s construction technique result in the worst case inputs for this tabulation problem.  Our algorithm can find such numbers, but we find it especially pleasing that they play a central role (as an obstacle!) in the tabulation algorithm.  

\subsection{Pinch's Easy Case}\label{subsec:bnd_lrg}

Using the congruence on $r-1$, we may take steps in an arithmetic progression to find candidate $r$ values.  The number of steps required is bounded by 
\[ k = \min \left\{  \frac{Pq - 1}{\lambda(Pq)}, \frac{B}{Pq\lambda(Pq)} \right\}.\]
The first term comes from the fact that $r$ is a factor of $Pq-1$ and the second term comes from the fact that $r < B/Pq$.  If either of these numbers is small, we take the $k$ steps in this arithmetic progression.  Recall our guidance that $\lambda(Pq)$ is close in size to $Pq$.  Therefore, it is common for the first ratio to be small.  In a similar way, when $Pq > B^{1/2}$, then the second ratio is expected to be small.  Most of the preproducts fall into one of these two cases and we consider this the bounded case.  For the rest of the section we assume both of these ratios are large.  Therefore, everything that follows is considered to be in the unbounded case.

\subsection{Pinch and Trial Division}\label{pinch_trail}

Richard Pinch notes that checking for candidates in arithmetic progression
\[r^{\star}, r^{\star} + \lambda(Pq), r^{\star} + 2\lambda(Pq), r^{\star} + 3\lambda(Pq), \ldots \]
may be sped up considerably by running over small factors $f$ of $Pq - 1$ to find candidates for $r$.  Put $r = (Pq - 1)/f + 1$, then test the congruence.  Testing small factors $f$ lowers the upper bound on the arithmetic progression.  

\begin{theorem} \label{pinch_large}
Finding $r$ in the manner described by Richard Pinch creates 
\[ O\left(\sqrt{\frac{Pq}{\lambda{(Pq)}}}\right) \]
candidates to check. 
\end{theorem}
\begin{proof}
We balance $k$ trial division steps with checking $k$ elements in arithmetic progression.  The number of candidates created is 
\[ O\left(k + \frac{Pq-1}{k\lambda(Pq)}\right).  \]
The first term of the sum represents the number of candidates by taking $k$ steps of trial division.  The second term counts the number of elements in arithmetic progression with common difference of $\lambda{(Pq)}$ that are bounded in size by $(Pq-1)/k$.  These two terms are balanced at 
\[k = \sqrt{\frac{Pq}{\lambda(Pq)}}. \]
\end{proof}

We follow Richard Pinch in spirit but show that there is a better way of finding the small factors $f$ that divide $Pq-1$.  This change provides an asymptotically superior algorithm.

\subsection{Divisors in Residue classes (Part 1)} \label{subsec:faster}
Let us rename quantities.  Let $g = \gcd( r^{\star} -1, \lambda(Pq))$.  Then, rather than search for $r-1$ directly, we will search for $r-1$ as a multiple of $g$.  Renaming quantities $\mathcal{R}_1 = (r^{\star} - 1)/g, \mathcal{S} = \lambda(Pq) /g,$ and $\mathcal{P} = (Pq-1)/g$, the problem may be restated:  Find factors of $\mathcal{P}$ that are in the residue class $\mathcal{R}_1 \pmod{\mathcal{S}}$.  This is the divisor in residue class problem.  We will explain and analyze what we did and how it compares with Pinch's approach.

The search for divisors $t$ of $\mathcal{P}$ may be broken into two cases.  Case 1:  $t \leq \sqrt{\mathcal{P}}$.  Case 2:  $t > \sqrt{\mathcal{P}}$.  In naive trial division, any successful division always yields two factors:  one in each of the two cases.  Since we are only concerned with divisors in specific residue classes, the two cases are treated independently.  

\subsubsection{Case 1} We construct divisors in the arithmetic progression
\[ \mathcal{R}_1, \mathcal{R}_1 + \mathcal{S}, \mathcal{R}_1 + 2 \mathcal{S}, \mathcal{R}_1 + 3\mathcal{S}, \mathcal{R}_1 + 4 \mathcal{S}\ldots \]
while these candidates remain less than or equal to $\sqrt{\mathcal{P}}$.  Multiply by $g$ and add $1$ to get the original arithmetic progression for $r$:  
\[ r^{\star}, r^{\star}+ 2\lambda(Pq) , r^{\star}+ 3\lambda(Pq),  \ldots.\]
This arithmetic progression does not deviate from the original approach.  The renaming is only necessary to compute the new upper bound.  

\subsubsection{Case 2}  Suppose we have $t >\sqrt{\mathcal{P}}$ and $t \equiv \mathcal{R}_1 \pmod{\mathcal{S}}$.  That means that there is a factor, say $t'$ with $t(t') = \mathcal{P}$ and $t' < \sqrt{\mathcal{P}}$.  This $t'$ lies in some residue class modulo $\mathcal{S}$.  If $(\mathcal{R}_1 + k_1\mathcal{S}) \mid \mathcal{P}$ then there exists $t' = \mathcal{R}_2 + k_2\mathcal{S}$ such that $(\mathcal{R}_1 + k_1\mathcal{S})(\mathcal{R}_2 + k_2\mathcal{S}) = \mathcal{P}$.  A simple computation shows that $\mathcal{R}_2 \equiv \mathcal{P}\mathcal{R}_1^{-1} \pmod{\mathcal{S}}$.  We construct divisors in the arithmetic progression
\[ \mathcal{R}_2, \mathcal{R}_2 +  \mathcal{S}, \mathcal{R}_2 +  2\mathcal{S}, \mathcal{R}_2 +  3\mathcal{S}, \mathcal{R}_2 +  4\mathcal{S}, \ldots \]
while these candidates remain less than $\sqrt{\mathcal{P}}$.  At each point, we check if $\mathcal{P} \equiv 0 \mod{(\mathcal{R}_2 +  k_2\mathcal{S} )}$.  If it is, then $t  = \mathcal{P}/ (\mathcal{R}_2 +  k_2\mathcal{S})$ is the divisor greater than $\sqrt{\mathcal{P}}$.   Multiply the quantity by $g$ and add $1$ to get $r$.  This is the same as searching for $r$ in 
\[ \frac{Pq-1}{\mathcal{R}_2} + 1,\frac{Pq-1}{\mathcal{R}_2 + \mathcal{S} } + 1  ,\frac{Pq-1}{\mathcal{R}_2 + 2\mathcal{S} } + 1 ,\frac{Pq-1}{\mathcal{R}_2 + 3\mathcal{S} } + 1, \ldots  \]
where the divisions need to be exact.  These denominators are a subset of the ``small factors $f$ of $Pq-1$" described by Pinch.  Because these denominators are constructed in arithmetic progression we have fewer candidate $r$ values to construct.  

\begin{theorem}\label{better_factoring}
Finding $r$ in the manner described above creates 
\[ O\left(\frac{\sqrt{gPq}}{\lambda{(Pq)}}\right) \]
candidates to check. 
\end{theorem}
\begin{proof}
Each arithmetic progression has the same form.  Therefore, the candidates created in each class is already balanced and all that remains is to count their number.  The inequality on $k$ bounds the number of candidates:     
\[ \mathcal{R} + k\mathcal{S} < \sqrt{\mathcal{P}} \mbox{ implies }  k < \frac{\sqrt{\mathcal{P}}}{\mathcal{S}} < \frac{\sqrt{gPq}}{\lambda{(Pq)}} .  \]
\end{proof}
Notice that the $\lambda{(Pq)}$ term is now outside of the radical at the expense of $g$ appearing in the radical.  

The sequence of numbers in Case 1 is increasing and the sequence of numbers in Case 2 is decreasing.  We may generate the numbers in Case 2 in increasing order by starting at the maximum allowed value of $k_2$ and working down to $0$.   Since we do not want to create factors that exceed $B/(Pq)$, the fact that this algorithm generates factors by size is an important feature.   It is straight-forward to implement an early-abort feature to avoid generating factors that are too big.  

This incorporates an insight of \cite{div_in_class, div_in_class_con} that if $\mathcal{S} > \mathcal{P}^{1/2}$, then there are only two possible divisors (those being $\mathcal{R}_1$ and $\mathcal{P}/\mathcal{R}_2$) and that they are found in polynomial time.  Given our guidance about the size of $\lambda(Pq)$, the expected case is for there to be only two possible divisors.  

\subsection{Divisors in Residue classes (Part 2)} \label{subsec:faster2}

If $\mathcal{S} > \mathcal{P}^{1/3}$, then there is also a polynomial time algorithm to find all possible factors that are $\mathcal{R}_1 \pmod{\mathcal{S}}$  based off of solving a logarithmic number of quadratic equations \cite{div_in_class}.  This result was improved in \cite{div_in_class_con}; they used LLL lattice reduction techniques to find factors in polynomial time if $\mathcal{S} > \mathcal{P}^{1/4}$.    In \cite{cranpom}, the authors of \cite{div_in_class_con} are reported to have said that it is practical for $\mathcal{S} > \mathcal{P}^{0.29}$.  

We implemented the version of \cite{div_in_class} described as Algorithm 4.2.11 in \cite{cranpom}.  A remark (changing the notation to be consistent with our own) found in \cite{cranpom} reads, ``If $\mathcal{S} < \mathcal{P}^{1/3}$, Algorithm 4.2.11 still works, but the number of square root steps is then $O(\mathcal{P}^{1/3}\mathcal{S}^{-1} \ln \mathcal{P}).$"  We do not believe this is true of a naive implementation of Algorithm 4.2.11.  The central claim involves the length of the interval $[ 2a_ib_i, a_ib_i + \mathcal{P}/\mathcal{S}^2]$.  This interval has length $O(\mathcal{P})$ when $\mathcal{S}$ is small.  We observed this nearly-linear run-time for especially small values of $\mathcal{S}$.  However, one may choose a multiplier $m$ so that $m\mathcal{S}$ is of the size $\mathcal{P}^{1/3}$.  Invoking Algorithm 4.2.11, $m$ times with modulus $m\mathcal{S}$ will give the result found in the remark. 

It is further noted that ``it remains an open question whether an efficient algorithm can be found that finds divisors of $\mathcal{P}$ that are congruent to $\mathcal{R}_1 \pmod{\mathcal{S}}$ when $\mathcal{S}$ is about $\mathcal{P}^{1/4}$ or smaller."  Our problem is, frequently, a bit more restricted than even this question.  We only want to find factors bounded by $B/Pq$.  The fastest algorithm we are aware of that can use both  residue class information and incorporate a bounded search would be a customized version of the Pollard-Strassen polynomial evaluation method. 

\subsection{Worst Case Inputs} \label{bad_problems}

Suppose $Pq$ is a Carmichael number.  In this case, $r^{\star} - 1 = 0$ and so we have no residue class information to use; the problem is to factor $(Pq-1)/\lambda(Pq)$.  By the Erd\H{o}s construction, it is possible to make Carmichael numbers with $\lambda(Pq)$ arbitrarily small as a power of $Pq$.  To make matters worse, $Pq$ may be a Carmichael number less than $B^{2/3}$.  In this case, there is no relevant bound information.  It is truly a general integer factorization problem.  Since this is a worst case input to our algorithm, it would make sense to analyse the asymptotic density of these numbers to know how often we encounter them.  On the one hand, we could consult current tabulations as a guide for how often we expect to have these inputs.  On the other hand, the asymptotic count remains a key open question in analytic number theory.

\section{Bound Admissible Preproducts} \label{sec:bndadmis}

In Sections \ref{sec:P_small} and \ref{sec:P_large}, we make the case for a bifurcated approach for tabulation.  In each case, we studied the problem as the following:  Given $P$ find $q$ and $r$ so that $Pqr$ is a Carmichael number.  In Section \ref{sec:P_small}, $q$ and $r$ are found simultaneously.  In Section \ref{sec:P_large}, we see that the tabulation bound $B$ implies that an exhaustive search for $q$, followed by a constructive search for $r$ is reasonable.  In both cases, this involves finding $P$ efficiently.  A bound admissible preproduct $P_k$ for $k < d$ is a cyclic number where $P_k p_k^{d-k} < B$.  In Section \ref{sec:analysis}, we show that due to the count of bound-admissible preproducts, care must be taken in generating these numbers.  As with the tabulation itself, our approach is bifurcated.    

\subsection{Bound admissible \texorpdfstring{$P$}{} } \label{subsec:P}
We need to ensure that $P$ is a cyclic number.  This is easy to do if we have access to the prime factorization of $P$.  We say $P$ is bound admissible if $Pp_{d-2}^2 < B$.  This means that any cyclic number less than $B^{1/3}$ will be bound admissible.  As we consider cyclic numbers greater than $B^{1/3}$, it becomes more difficult for a cyclic number to be bound admissible.  For instance, no prime number greater than $B^{1/3}$ is bound admissible.  

Given the relatively slow growth of $\log \log \log B$ and the need to have factorizations of all numbers through $2P-1$ in Section \ref{sec:P_small}, a variant of the sieve of Eratosthenes that factors all numbers was used.  The cost of invoking this is a lower-order asymptotic cost in the overall run-time of this portion of the algorithm.  

\subsection{Bound admissible \texorpdfstring{$P_{d-1}$}{}}\label{subsec:Pdm1}
Assume all $P \leq X$ were accounted as small.  Then a bound admissible $P_{d-1} = Pq$ satisfies $Pq^2 < B$ and $P > X$.  We generate $P$ and $q$ from a table of primes.  Let $P_{d-1}$ have $k > 2$ prime factors.  We know $k \not = 2$, because that case was considered as small.  For the $d-2$ factors that comprise the primes in $P$, we need to ensure at each level that the product is bound-admissible, cyclic, and will be greater than $X$.  For the $(d-1)^{\rm{th}}$ prime $q$, we ensure the product is still cyclic and is bound admissible.   

In Section \ref{sec:P_large}, we showed that $q$ is in the interval $(p_{d-2}, \sqrt{B/P}) \subset (p_{d-2}, \sqrt{B/X})$.  If $X$ is chosen to be at least $B^{1/3}$, that means the table will only require primes up to $B^{1/3}$.  Naturally, the larger we choose $X$, the smaller the table needs to be.  

At our scale, a significant portion of the total tabulation time of the  large case was spent just in preproduct creation.  It was important to both minimize duplicated work and yet have a balanced, efficient parallel algorithm.  To properly distribute the computation of preproducts (and thus the work to find the completions of such preproducts) with $k$ prime factors, every processor created all bound-admissible preproducts with $\lfloor k/2 \rfloor$ prime factors.  These were counted and then a given processor advanced to the next appending more primes only when the count of preproducts is in its residue class modulo the total number of processors.   

\section{Asymptotic counts of Preproducts}\label{sec:analysis}

Due to empirical timing results and the theoretical results of Section \ref{sec:P_large}, it is reasonable to believe that the most frequent case of completing a preproduct $Pq$ is done in polynomial time.  Therefore, a reasonable guide to the cost of tabulation is the count of bounds admissible preproducts.    
\begin{theorem}\label{num_preprod}
The number of bound admissible $P_{d-1}$ is bounded above by
\[ B\exp\left(\left(-\frac{1}{\sqrt{2}} +o(1)\right)\sqrt{\log B \log\log B}\right). \]
\end{theorem}

\begin{proof}

Letting $P_{d-1} = P_{d-2}p_{d-1}$, then $P_{d-2}p_{d-1}^2 < B$.  We will count $P_{d-2}$ as a number less than $B/p_{d-2}^2$ and being $p_{d-2}$-smooth.  That is, the number of $P_{d-1}$ under consideration is at most
\[\sum_{p<B^{1/2}} \Psi\left(\frac{B}{p^2},p\right) \]
where $\Psi(x,y)$ is the number of $y$-smooth numbers up to $x$.

We split the sum $p\leq \exp(c\sqrt{\log B \log\log B})$ and $\exp(c\sqrt{\log B \log\log B})<p<B^{1/2}$ where $c$ is a constant to be chosen later. The former range is treated by
\[ \sum_{p\leq \exp(c\sqrt{\log B \log\log B})} \Psi\left(\frac{B}{p^2},p\right) \]
\[\leq\sum_{p\leq\exp(c\sqrt{\log B \log\log B})} \Psi(B,\exp(c\sqrt{\log B \log\log B}) ) \]
\[ \leq B\exp\left(\left(c-\frac1{c}+o(1)\right)\sqrt{\log B \log\log B}\right) \]
by Hildebrand's estimate on smooth numbers \cite{smooth_hild86}. 


The latter range is treated by the trivial bound
\[ \sum_{\exp(c\sqrt{\log B \log\log B})<p\leq B^{1/2}} \frac{B}{p^2}\ll B\exp(-c \sqrt{\log B \log\log B}).\]

The optimal $c$ is satisfying $c-\frac{1}{c}=-c$, so $c=1/\sqrt{2}$.
\end{proof}

We were initially concerned with the count of preproducts from the small case.  We posed that problem in a session at the West Coast Number Theory Conference (2022) and Sungjin Kim showed us how to count that set.  In particular, he proved the following.

\begin{theorem}
The number of bound admissible $P_{d-2}$ is bounded above by
\[ B\exp\left(\left(-\frac{2}{\sqrt{6}} +o(1)\right)\sqrt{\log B \log\log B}\right). \]
\end{theorem}

The proof is substantially the same and we thank him for his help on this problem.  

The count given in Theorem \ref{num_preprod} differs from the set we used in two ways.  First, we counted smooth numbers instead of counting cyclic numbers.  However, since the asymptotic density of cyclic numbers is only off from linear by a $\log \log \log B$ factor, we do not believe that the difference in these two sets is meaningful.  Second,  we have not omitted small preproducts from consideration.   We would only want to count bound admissible $P_{d-1}$ for which $P_{d-2} > B^{1/3}$.  However, we also do not believe that this set would be significantly different from what we counted in Theorem \ref{num_preprod}.


We believe, as a guide, that the cost of completing a given preproduct $Pq$ is a polynomial-time algorithm, and that the bound in Theorem \ref{num_preprod} serves as a guide for the count of preproducts constructed.  A better guide might tell us the distribution of preproducts in $[ 0, B]$.  In Section \ref{sec:P_large}, we saw that the larger a preproduct was, the more likely it was to be an easy case.    So, both the number of difficult problems and their distribution is important to establish sharper asymptotic bounds.  

\section{Implementation, statistics, and conclusion}  \label{sec:conc}

While our high level overview in Section \ref{sec:alg_over} presented tabulation as a single algorithm, the implementation consisted of many different programs.   We chose $X = 7 \cdot 10^7$ as our cross-over point and $B = 10^{22}$.  The choice of $X$ was made before we knew what would be a feasible choice for $B$.  The small and large preproduct cases were implemented and run independently.  The large case was written so that it finds Carmichael numbers with exactly $d$ prime factors and had many separate computations:  one for each $4 \leq d \leq 9 $ and one recursive version that handled $d > 9$.  We implemented in C++;  everything was done in machine word size except three quantities that used GMP.  In the small case, we constructed Carmichael numbers outside the machine word bounds and also used GMP for primality testing of $q$ and $r$.  The code may be found on GitHub \cite{TabCodeSW}.  

Combining large and small cases provides the exact count by prime factors for the total tabulation up to $B = 10^{22}$.
\[
\begin{array}{|c|c|c|c|c|c|} \hline

d &  3& 4 & 5  & 6  & 7 \\ \hline 
C_d(10^{22})& 420658 & 232742  & 1370257 & 6034046  & 14056367  \\ \hline \hline  
 8 & 9  & 10  & 11  & 12  & 13 \\ \hline
 16005646 &  8939435  & 2347828  & 262818  & 10018 & 55 \\ \hline
\end{array}
\]
Each of these numbers was proven to be a Carmichael number with a separate check program, and we confirmed that the numbers up to $10^{21}$ matched Pinch's most recent tabulation.  Testing was also aided by a comparison against an unpublished tabulation performed by Claude Goutier.

The next two subsections give greater details about the two different computational regimes.

\subsection{Small Preproduct Information}

In \cite{sw_ants2}, the computation was meant to be a proof of concept and had a relatively small bound of $P < 3 \cdot 10^6$.  We have extended the bound to $P < 7 \cdot 10^7$.   In Theorem 5 of \cite{sw_ants2}, it was shown that the total cost to use all $P < X$ is $O( (X \log X)^2 )$.  The small case tabulation here represents about $800$ times more work over that of \cite{sw_ants2}.  It was computed using 33.1 vCPU years on AMD EPYC 7B13 CPUs\footnote{Special thanks to Drew Sutherland and the Simons Collaboration in Arithmetic Geometry, Number Theory, and Computation for computational support.}.

We did the computation unbounded and found $35985331$ Carmichael numbers.  Of these, only $1202914$ were less than $B$.  In the table below, we show the contribution to the tabulation by order of magnitude:

\[
\begin{array}{|c|c|c|c|c|c|c|} \hline
10^3 & 10^4 & 10^5 & 10^6 & 10^7 & 10^8 & 10^9  \\ \hline
1 & 7 & 16 & 43 & 105 & 255 & 646 \\   \hline  \hline

10^{10} & 10^{11} & 10^{12} & 10^{13} & 10^{14} & 10^{15} & 10^{16}  \\ \hline
1547 & 3605 & 8238 & 19019 & 40772 & 79417 & 137794  \\ \hline \hline

10^{17} & 10^{18} & 10^{19} & 10^{20} & 10^{21} & 10^{22} & 10^{23}  \\ \hline
219887 & 327732 & 464605 & 639763 & 872861 & 1202914 & 1708687  \\ \hline
\end{array}
\]

It should be noted that this produced a complete tabulation of all Carmichael numbers less than $10^{11}$.  For the $10^{12}$ bound, this computation only missed three Carmichael numbers:
\begin{itemize}
\item $702712420201 = ( 11 \cdot 37 \cdot 43 \cdot 61 \cdot 71 ) \cdot 73 \cdot 127$,
\item $919707221161 = ( 19 \cdot 29 \cdot 43 \cdot 53 \cdot 73 ) \cdot 79 \cdot 127$,
\item $995483689201 = ( 19 \cdot 31 \cdot 37 \cdot 53 \cdot 61 ) \cdot 71 \cdot 199$.
\end{itemize}
Each of these numbers has a preproduct (the portion in parentheses) that exceeds $7 \cdot 10^7$.  So, this tabulation regime could not have found these numbers, but they were found in the large preproduct portion of the tabulation with $d=7$.  For the $10^{13}$ bound, there are $260$ missing Carmichael numbers.  And so it goes; the small case made the most significant contributions with smaller $d$ values.  For the cases $d \geq 7$, the small case found no Carmichael numbers.  

Since we ran this part of the computation as unbounded, we found many Carmichael numbers that were significantly larger than $10^{22}$.  The largest one we found is
\begin{align*}
69999133 \cdot 4899878690750821 \cdot 171493630078866294519097 = \\
 58 \ 82013 \ 03152 \ 54068 \ 53935 \ 58087 \ 37155 \ 82013 \ 87008 \ 71721. 
 \end{align*}
This example may be reproduced with the parameters $P = 69999133,  D = 2,$ and $\Delta = 1$ and would have been found with the $D\Delta$ method.   The dynamic choice in the hybrid method would have chosen the lesser amount of work:  iterate through $768$ divisors or iterate through approximately $7 \cdot 10^7$ values of $C$.  This example represents $q$ and $r$ taking maximal values.  That is, $ q - 1 = ( P - 1 )( P + 2 )$ and  $ r - 1 = ( Pq - 1 ) / 2 $.


\subsection{Large Preproduct Information}

Here are the counts of Carmichael numbers constructed in the large preproduct case, broken down by $d$: 
\[
\begin{array}{|c|c|c|c|c|} \hline
4 & 5  & 6  & 7  & 8 \\ \hline 
100164  & 1024718 & 5769410  & 14017450  & 16005060\\ \hline \hline  
 9  & 10  & 11  & 12  & 13 \\ \hline
 8939435  & 2347828  & 262818  & 10018 & 55 \\ \hline
\end{array}
\]

The following table gives timing information for the $d=6$ tabulation.  Method $1$ is what we propose in Subsection \ref{subsec:faster} which uses the simplest example of the divisors in residue class results.  Method $2$ is what is described in Subsection \ref{pinch_trail} and balances sieving against naive trial division.     All timings are given in seconds run on a single core of an AMD Ryzen 9 3900X.

\[
\begin{array}{|c|c|c|c|c|c|c|} \hline
B & \mbox{Method 1 (SW)} & \mbox{Method 2 (Pinch)}  \\ \hline
10^{14} & 25 & 25  \\ \hline
10^{15} & 192 & 189 \\ \hline
10^{16} & 1324 & 1306   \\ \hline
10^{17} & 8752 & 8865   \\ \hline
10^{18} & 55631 & 56072  \\ \hline
10^{19} & 361983 & 364816 \\ \hline
\end{array}
\]
The improvements are very modest.  The reason is that the proportion of cases for which the new method gives an exponential speedup are relatively rare.    We know that the new method outperforms Pinch's at asymptotic scales, but unfortunately the scales at which computations are performed are not asymptotic.  One observation that gives hope is that the timings of Method 1 got better as the upper bound increased. 
 It is not too hard to pick out isolated examples where the improve algorithm is demonstrably better.  For example, consider $P = 11756813235$.  The timings measured were $9$ microseconds for Method 1, and $17$ microseconds for Method 2.  

\subsection{Future Work}

We intend to continue to work on this problem.  Our guidance is that completing a preproduct is done on average in polynomial time;  thus the primary cost of tabulation was constructing the preproducts.  It would be better if the cost to enter the inner loop was insignificant compared to the work done in the inner loop.   We believe there are faster ways of constructing preproducts and that we may consider fewer preproducts.  Our goal is to tabulate all Carmichael numbers less than $10^{23}$. 

\bibliographystyle{amsplain}
\bibliography{manybases}

\providecommand{\bysame}{\leavevmode\hbox to3em{\hrulefill}\thinspace}
\providecommand{\MR}{\relax\ifhmode\unskip\space\fi MR }
\providecommand{\MRhref}[2]{%
  \href{http://www.ams.org/mathscinet-getitem?mr=#1}{#2}
}
\providecommand{\href}[2]{#2}
\begin{thebibliography}{10}

\bibitem{subset}
W.~R. Alford, Jon Grantham, Steven Hayman, and Andrew Shallue,
  \emph{Constructing {C}armichael numbers through improved subset-product
  algorithms}, Math. Comp. \textbf{83} (2014), no.~286, 899--915. \MR{3143697}

\bibitem{inf_carm}
W.~R. Alford, Andrew Granville, and Carl Pomerance, \emph{There are infinitely
  many {C}armichael numbers}, Ann. of Math. (2) \textbf{139} (1994), no.~3,
  703--722. \MR{1283874 (95k:11114)}

\bibitem{div_in_class_con}
Don Coppersmith, Nick Howgrave-Graham, and S.~V. Nagaraj, \emph{Divisors in
  residue classes, constructively}, Math. Comp. \textbf{77} (2008), no.~261,
  531--545. \MR{2353965}

\bibitem{cranpom}
Richard Crandall and Carl Pomerance, \emph{Prime numbers}, second ed.,
  Springer, New York, 2005, A computational perspective. \MR{2156291}

\bibitem{lambda_cite}
Paul Erd\H{o}s, Carl Pomerance, and Eric Schmutz, \emph{Carmichael's lambda
  function}, Acta Arith. \textbf{58} (1991), no.~4, 363--385. \MR{1121092}

\bibitem{erdos_cyclic}
P.~Erd\"{o}s, \emph{Some asymptotic formulas in number theory}, J. Indian Math.
  Soc. (N.S.) \textbf{12} (1948), 75--78. \MR{29406}

\bibitem{lambda_cite2}
John~B. Friedlander, Carl Pomerance, and Igor~E. Shparlinski, \emph{Period of
  the power generator and small values of {C}armichael's function}, Math. Comp.
  \textbf{70} (2001), no.~236, 1591--1605. \MR{1836921}

\bibitem{two_contra}
Andrew Granville and Carl Pomerance, \emph{Two contradictory conjectures
  concerning {C}armichael numbers}, Math. Comp. \textbf{71} (2002), no.~238,
  883--908. \MR{1885636}

\bibitem{guthmann}
A.~Guthmann, \emph{On the computation of carmichael numbers}, Tech. Report 218,
  Universität Kaiserslautern, 1992,
  \url{https://kluedo.ub.rptu.de/frontdoor/deliver/index/docId/5039/file/Guthmann_On+the+computation+of+Carmichael+numbers.pdf}.

\bibitem{better_inf}
Glyn Harman, \emph{Watt's mean value theorem and {C}armichael numbers}, Int. J.
  Number Theory \textbf{4} (2008), no.~2, 241--248. \MR{2404800}

\bibitem{smooth_hild86}
Adolf Hildebrand, \emph{On the number of positive integers {$\leq x$} and free
  of prime factors {$>y$}}, J. Number Theory \textbf{22} (1986), no.~3,
  289--307. \MR{831874}

\bibitem{jaeschke_carm}
Gerhard Jaeschke, \emph{The {C}armichael numbers to {$10^{12}$}}, Math. Comp.
  \textbf{55} (1990), no.~191, 383--389. \MR{1023763}

\bibitem{keller}
W.~Keller, \emph{The {C}armichael numbers to {$10^{13}$}}, Abstracts Amer.
  Math. Soc. \textbf{9} (1988), 328--329.

\bibitem{div_in_class}
H.~W. Lenstra, Jr., \emph{Divisors in residue classes}, Math. Comp. \textbf{42}
  (1984), no.~165, 331--340. \MR{726007}

\bibitem{pinch15}
R.~G.~E. Pinch, \emph{The {C}armichael numbers up to {$10^{15}$}}, Math. Comp.
  \textbf{61} (1993), no.~203, 381--391. \MR{1202611}

\bibitem{pinch16}
\bysame, \emph{The {C}armichael numbers up to {$10^{16}$}}, March, 1998,
  \texttt{arXiv:math.NT/9803082}.

\bibitem{pinch17}
\bysame, \emph{The {C}armichael numbers up to {$10^{17}$}}, April, 2005,
  \texttt{arXiv:math.NT/0504119}.

\bibitem{pinch18}
\bysame, \emph{The {C}armichael numbers up to {$10^{18}$}}, April, 2006,
  \texttt{arXiv:math.NT/0604376}.

\bibitem{pinch21}
\bysame, \emph{The {C}armichael numbers up to {$10^{21}$}}, May, 2007,
  \texttt{s369624816.websitehome.co.uk/rgep/p82.pdf}.

\bibitem{pollack_cyclic}
Paul Pollack, \emph{Numbers which are orders only of cyclic groups}, Proc.
  Amer. Math. Soc. \textbf{150} (2022), no.~2, 515--524. \MR{4356164}

\bibitem{carmsurvey}
Carl Pomerance, \emph{Carmichael numbers}, Nieuw Arch. Wisk. (4) \textbf{11}
  (1993), no.~3, 199--209. \MR{1251482}

\bibitem{pom_self}
Carl Pomerance, J.~L. Selfridge, and Samuel~S. Wagstaff, Jr., \emph{The
  pseudoprimes to {$25\cdot 10^{9}$}}, Math. Comp. \textbf{35} (1980), no.~151,
  1003--1026. \MR{572872 (82g:10030)}

\bibitem{TabCodeSW}
Andrew Shallue, \emph{Code for tabulating carmichael numbers},
  \url{https://github.com/ashallue/tabulate_car/tree/master}, 2024.

\bibitem{sw_ants2}
Andrew Shallue and Jonathan Webster, \emph{Tabulating {C}armichael numbers {$n
  = Pqr$} with small {P}}, Res. Number Theory \textbf{8} (2022), no.~4, Paper
  No. 84. \MR{4493784}

\bibitem{swift}
J.D. Swift, \emph{Review 13 -- table of {C}armichael numbers to {$10^{9}$}},
  Math. Comp. \textbf{29} (1975), no.~129, 338--339.

\bibitem{first_source}
V.~\u{S}imerka, \emph{Zbytky z arithemetick \`e posloupnosti (on the remainders
  of an arithmetic progression)}, Casopis pro p \u{e}stov \`an \`i matematiky a
  fysiky \textbf{14} (1885), 221--225.

\end{thebibliography}

\end{document}